\documentclass{amsart}
\usepackage{amssymb}
\usepackage{amsmath}
\usepackage{amsthm}
\usepackage{graphicx}
\usepackage{float}
\usepackage{color}
\usepackage{csquotes}

\newcommand{\re}{\mathop{\mathrm{Re}}}
\newcommand{\im}{\mathop{\mathrm{Im}}}
\newcommand{\HQR}{\mathop{\mathrm{HQR}}}
\newcommand{\Hol}{\mathop{\mathrm{Hol}}}
\newcommand{\Har}{\mathop{\mathrm{Har}}}

\newcommand{\artanh}{\mathop{\mathrm{artanh}}}

\newtheorem{theorem}{Theorem}
\newtheorem{lemma}{Lemma}
\newtheorem{example}{Example}
\newtheorem{proposition}{Proposition}
\newtheorem{remark}{Remark}
\newtheorem{definition}{Definition}
\author{Miodrag Mateljevi\' c  \&  Marek Svetlik}
\address{Faculty of mathematics, University of Belgrade, Studentski Trg 16,
Belgrade, Republic of Serbia} \email{miodrag@matf.bg.ac.rs} \email{svetlik@matf.bg.ac.rs}
\title[The Schwarz lemma for HQR mappings]{Hyperbolic metric on the strip and the Schwarz lemma for HQR mappings}

\date{August 20, 2018}



\thanks{Research partially supported by Ministry of Education, Science and Technological Development
of the Republic of Serbia, Grant No. 174 032.}

\begin{document}

\maketitle

\begin{abstract}
We give simple proofs of various versions of the Schwarz lemma
for real valued harmonic functions and for holomorphic (more generally harmonic quasi\-re\-gu\-lar, shortly  HQR) mappings
with the strip codomain. Along the way  using the principle of subordination and the corresponding
conformal mapping, depicted  on the Figure \ref{fig:1}, we get a simple
proof of a new  version of the Schwarz lemma for  real valued harmonic functions (see Theorems
\ref{th:schwhar}  and  \ref{th:schwhar1}) and   Theorem~\ref{th:schwhol} related to  holomorphic mappings.
Using  the Schwarz-Pick lemma related to distortion   for harmonic mappings  and the elementary properties of  the hyperbolic
geometry  of the strip we prove Lemma \ref{lem:hqr}, which is a key ingredient  in
the proof of  Theorem  \ref{th:schwhqr}  which  yields  optimal estimates  for modulus of  HQR  mappings.

\end{abstract}

\section{Introduction and preliminaries}


Motivated by the role of the Schwarz lemma in complex analysis and
numerous fundamental results (see \cite{Ahl,osserman,krantz,MMSchw_Kob} and
 references cited there and for some recent result which are in our
research  direction \cite{AzerOr0,kavu,Khal,MMar,zhu}), in 2016, cf.
\cite{RgSchw1}(a), the first author has posted the current research project \enquote{Schwarz lemma, the Carath\'{e}odory and Kobayashi Metrics and Applications in Complex Analysis}\footnote{Motivated by  S. G. Krantz  paper \cite{krantz}.}.  Various discussions regarding the subject can also  be found  in the Q\&A section on ResearchGate under the question \enquote{What are the most recent versions of the Schwarz lemma ?},\cite{RgSchw1}(b)\footnote{The subject has been presented at Belgrade analysis seminar \cite{Bg_sem}.}. In this project  and  in  \cite{MMSchw_Kob}, cf. also \cite{kavu} we developed  the method  related to  holomorphic mappings with strip codomain (we refer to this method as  the approach  via the Schwarz-Pick lemma for holomorphic maps from the unit disc into a strip). Note here that our use of terms the Schwarz lemma and the Schwarz-Pick lemma is refer to the corresponding   versions for modulus and hyperbolic distances, respectively (we follow the terminology used in \cite{BeardonMinda}).

In particular  our work here  is related  to  previous  works   \cite{mat2,MaK},  and some recent results of Kalaj and  Vuorinen \cite{kavu} (shortly KV-results; see also D. Khavinson  \cite{Khav}, G. Kresin and V. Maz'ya \cite{Maz}).
As we mentioned in \cite{MMSchw_Kob}, it seems that  KV-results  influenced  further  research by H. Chen \cite{hhChen}, M. Markovi\'c \cite{MMar}, A. Khalfallah \cite{Khal} and P. Melentijevi\'c \cite{pmel}.

One of the purpose  of this paper, which  is a relatively elementary contribution  and  continuation of these research, is to demonstrate our approach and make a common frame for previous works.

Throughout this paper by $\mathbb{U}$ we denote the unit disc $\{z\in\mathbb{C}:|z|<1\}$. By the Riemann mapping theorem simply connected plane domains different from  $\mathbb{C}$ (we call these domains hyperbolic) are conformally equivalent to $\mathbb{U}$. Accordingly if $\Omega$ is a hyperbolic domain then by $\rho_{\Omega}(z)|dz|$ we denote the hyperbolic metric of $\Omega$. This metric induces a hyperbolic distance on $\Omega$ in the following way
\begin{equation*}
    d_{\Omega}(z_1,z_2)=\inf\int_{\gamma}\rho_{\Omega}(z)|dz|,
\end{equation*}
where the infimum is taken over all $C^1$ curves $\gamma$ joining $z_1$ to $z_2$ in $\Omega$.

It is well known that $\displaystyle\rho_{\mathbb{U}}(z)|dz|=\frac{2|dz|}{1-|z|^2}$ and immediately follows that for all $z_1,z_2\in\mathbb{U}$ it holds
\begin{equation*}
    d_{\mathbb{U}}(z_1,z_2)=\ln\frac{1+\sigma_{\mathbb{U}}(z_1,z_2)}{1-\sigma_{\mathbb{U}}(z_1,z_2)}=2\artanh{\sigma_{\mathbb{U}}(z_1,z_2)},
\end{equation*}
where the pseudo-hyperbolic distance $\sigma_{\mathbb{U}}$ is given by $\displaystyle\sigma_{\mathbb{U}}(z_1,z_2)=\left|\frac{z_1-z_2}{1-z_1\overline{z_2}}\right|$.

If $f$ is a conformal map from hyperbolic domain $\Omega$ onto $\mathbb{U}$ then the hyperbolic metric $\rho_{\Omega}(z)|dz|$ of $\Omega$ is defined by $\rho_{\Omega}(z)|dz|=\rho_{\mathbb{U}}(f(z))|f'(z)|$. Hence, one can transfer the concept of the hyperbolic distance from $\mathbb{U}$ on hyperbolic domain $\Omega$.

For more details related to hyperbolic domains, hyperbolic metric and distance, see, for example \cite{Ahl,BeardonMinda,MMTopics}.

In this paper, except the disc $\mathbb{U}$, of other hyperbolic domains we will mainly use the strip  $\mathbb{S}=\{z\in\mathbb{C}:-1<\re{z}<1\}$.

Let   $D,G$ be  domains in  $\mathbb{C}$. By $\Hol(D,G)$ (respectively $\Har(D,G)$) we denote the set of the all holomorphic (respectively harmonic) mappings $f:D\rightarrow G$.

By $d_e$ we denote Euclidean distance in $\mathbb{C}$ and for $z\in\mathbb{C}$ we define the functions $e$ and $R_e$,  by  $e(z)=d_e(0,z)=|z|$ and $R_e(z)= \re{z}$, respectively.

%

For completeness we first give the classical Schwarz lemma  which is a direct  corollary of maximum modulus principle.

\begin{theorem}[The classical Schwarz lemma - the Schwarz lemma for holomorphic maps from $\mathbb{U}$ into $\mathbb{U}$]\label{th:schwclassic}
Let $f\in\Hol(\mathbb{U},\mathbb{U})$ and $f(0)=0$. Then
\begin{equation}\label{schwclassic:fla1}
    |f(z)|\leqslant|z|, \quad \mbox{ for all } \quad z\in\mathbb{U}.
\end{equation}
and
\begin{equation}\label{schwclassic:fla2}
    |f'(0)|\leqslant1.
\end{equation}
In (\ref{schwclassic:fla1}) the equality holds for one $z\in\mathbb{U}-\{0\}$ and in (\ref{schwclassic:fla2}) the equality holds if and only if $f(z)=\alpha z$, where $\alpha\in\mathbb{C}$ such that $|\alpha|=1$.
\end{theorem}

The following theorem is known as the Schwarz lemma for harmonic maps from $\mathbb{U}$ into itself.

\begin{theorem}[The Schwarz lemma for harmonic maps from $\mathbb{U}$ into $\mathbb{U}$, \cite{heinz}, {\cite[p. 77]{duren}}]\label{th:schwharcmplx}
Let $f\in\Har(\mathbb{U},\mathbb{U})$ and $f(0)=0$. Then
\begin{equation}\label{schwharcmplx:fla1}
    |f(z)|\leqslant\frac{4}{\pi}\arctan{|z|}, \quad \mbox{ for all } \quad z\in\mathbb{U},
\end{equation}
and this inequality is sharp for each point $z\in\mathbb{U}$.
\end{theorem}


Using the concept of the hyperbolic metric and hyperbolic distance on hyperbolic domains one can derive the Schwarz-Pick lemma for simply connected domains as a corollary of the classical Schwarz lemma:


\begin{theorem}[The Schwarz-Pick lemma for simply connected domains, {\cite[Theorem 6.4.]{BeardonMinda}}]\label{th:schwpick}
Let $\Omega_1$ and $\Omega_2$ be hyperbolic domains and $f\in\Hol(\Omega_1,\Omega_2)$. Then
\begin{equation}\label{schwpick:fla1}
    \rho_{\Omega_2}(f(z))|f'(z)|\leqslant\rho_{\Omega_1}(z), \quad \mbox{ for all } \quad z\in\Omega_1.
\end{equation}
and
\begin{equation}\label{schwpick:fla2}
    d_{\Omega_2}(f(z_1),f(z_2))\leqslant d_{\Omega_1}(z_1,z_2), \quad \mbox{ for all } \quad z_1,z_2\in\Omega_1.
\end{equation}
In (\ref{schwpick:fla1}) and (\ref{schwpick:fla2}) the equalities hold if and only if $f$ is a conformal isomorphism from $\Omega_1$ into $\Omega_2$.
\end{theorem}
In this paper we will use only the special case of this result if the domain is the unit disc  and  the  codomain is the strip.\\



In  Example \ref{exa:phi} we consider the conformal mapping $\phi$ from the unit disc $\mathbb{U}$ onto the strip $\mathbb{S}$. In  Lemmas \ref{lem:hypdisk}  and \ref{lem:hypdisk1}  we explicitly find  the maximum and minimum of the function  $R_e$ and the maximum of the function $e$ on the closed hyperbolic disc in $\mathbb{S}$ which is obtained as image of the closed hyperbolic disc with center $0$ in $\mathbb{U}$ by mapping $\phi$. Note that the proof of Lemma \ref{lem:hypdisk} is elementary  and it is based on the properties of mapping $\phi$. The proof of Lemma \ref{lem:hypdisk1} is based on Proposition \ref{prop:strip} which gives us an interesting relation between the hyperbolic and Euclidean distance on $\mathbb{S}$. Theorem \ref{th:schwhar} follows directly  from the formula (\ref{hypdisk:fla0}) of Lemma \ref{lem:hypdisk}  and the subordination principle. Further development  of this  method  yields  Theorem \ref{th:schwhar1} (without hypothesis that $0$ is mapped to $0$) which seems to be a new result (see also Example \ref{exa:phib} and Lemma \ref{lem:hypdisk0}).

It seems here  that it is right place to emphasize the following difference between holomorphic and harmonic maps.

If $f$ is holomorphic mapping from $\mathbb{U}$ into itself such that $f(0)=b$, where $b\in\mathbb{U}$, then using the mapping $f^b=\varphi_b \circ f$,
where $\varphi_b$ is conformal automorphism of $\mathbb{U}$ (see Example \ref{exa:varphia} below), we reduce this situation to the case $b=0$, since $f^b(0)=0$. As far as we know the researchers have some difficulties  to handle the case  $f(0)=b$  if  $f$ is harmonic mapping from $\mathbb{U}$ into $(-1,1)$,  since in that case the mapping  $f^b$  is not harmonic in general. Our method overcome this difficulty.

To get optimal estimate  for modulus of holomorphic
(more generally HQR)  mappings  we use
the elementary properties of  the hyperbolic
geometry  of the strip, see  Lemmas \ref{lem:hypdisk1}  and  \ref{lem:hqr}, and   Theorems \ref{th:schwhol} and  \ref{th:schwhqr}.

In order to establish Theorem \ref{th:schwhqr} (the Schwarz lemma for HQR maps from $\mathbb{U}$ into $\mathbb{S}$) among other things we will use the following elementary considerations (see \cite{MMSchw_Kob}):

\begin{itemize}
\item[(I)] Suppose  that $f\in\Hol(\mathbb{U},\mathbb{S})$. Then by Theorem \ref{th:schwpick} we have $\rho_{\mathbb{S}}(f(z))|f'(z)|\leqslant\rho_{\mathbb{U}}(z)$, for all $z\in \mathbb{U}$.
\item[(II)] If $f=u+iv$ is  a  complex valued  harmonic and $F=U+iV$ is a holomorphic function on a domain  $D$   such that $\re{f}=\re{F}$  on  $D$ (in this setting we say that $F$ is associated to $f$ or to $u$), then $F'=U_x+iV_x = U_x-iU_y = u_x-iu_y$. Hence, if $\nabla u=(u_x,u_y)=u_x+iu_y$ then $F'=\overline{\nabla u}$ and $|F'|=|\overline{\nabla u}|=|\nabla u|$.
\item[(III)] Suppose that $D$ is a simply connected plane domain and $f:D \rightarrow \mathbb{S}$ is a complex valued harmonic function. Then it is known from the standard course of complex analysis that there is a holomorphic function $F$ on $D$  such that  $\re{f} = \re{F}$ on $D$, and it is clear that $F:D\rightarrow\mathbb{S}$.
\item[(IV)] The hyperbolic density $\rho_{\mathbb{S}}$ at point $z$  depends only on  $\re{z}$.
\end{itemize}
By  (I)-(IV)  it is readable  that   we have
\begin{proposition}[{\cite[Proposition 2.4]{MMSchw_Kob}}, \cite{kavu,hhChen}]\label{prop:mm}
Let $u:\mathbb{U}\rightarrow(-1,1)$ be harmonic function and let $F$ be holomorphic function which  is associated to $u$. Then
\begin{equation}\label{lemhqr:fla3}
    \rho_{\mathbb{S}}(F(z))|\nabla u(z)|\leqslant \rho_{\mathbb{U}}(z)\quad \mbox{ for all }\quad z\in\mathbb{U}.
\end{equation}
\end{proposition}
Note  the above described  simple  method
basically    based  on  the Schwarz-Pick lemma for holomorphic maps from $\mathbb{U}$ into $\mathbb{S}$ yields a proof of  the above proposition to which we  refer  as the Schwarz-Pick lemma related to distortion for harmonic maps from $\mathbb{U}$ into $(-1,1)$. By this proposition  we control  distortion  of HQR mappings.

Note here that  there is tightly connection between the subordination principle and the various versions of the Schwarz-Pick lemma for holomorphic maps from $\mathbb{U}$ into $\mathbb{S}$.
Namely  in the proof of Theorem \ref{th:schwhar} and Theorem \ref{th:schwhol} (the Schwarz lemma for holomorphic maps from $\mathbb{U}$ into $\mathbb{S}$) we have used a corollary  of the subordination principle which can be stated in the form (see Definition \ref{def1} for notation):
\begin{itemize}
\item[(V)] If  $f\in\Hol(\mathbb{U},\mathbb{S})$  and $a\in \mathbb{U}$,  then the image  of the hyperbolic disc in $\mathbb{U}$  with hyperbolic center $a$ and hyperbolic radius $\lambda$ under $f$ is in the hyperbolic disc in  $\mathbb{S}$ with hyperbolic center at $f(a)$ and hyperbolic radius $\lambda$.
\end{itemize}
Note that (V) is the Schwarz-Pick lemma for holomorphic maps from $\mathbb{U}$ into $\mathbb{S}$.
Instead of subordination principle in the proof of Theorem \ref{th:schwhqr} (the Schwarz lemma for HQR maps from $\mathbb{U}$ into $\mathbb{S}$)  we use Lemma  \ref{lem:hqr}, which can be consider as a generalization of (V):
\begin{itemize}
\item[(VI)] If  $f\in\HQR_K(\mathbb{U},\mathbb{S})$ (see definition below) and $a\in \mathbb{U}$, then the image  the hyperbolic disc in $\mathbb{U}$  with hyperbolic center $a$ and hyperbolic radius $\lambda$ under $f$ is in the hyperbolic disc in  $\mathbb{S}$ with hyperbolic center $f(a)$ and hyperbolic radius $K \lambda$.
\end{itemize}

Note that proof of  Lemma  \ref{lem:hqr} is based on   Proposition  \ref{prop:mm}   (the Schwarz-Pick lemma related to distortion    for harmonic maps  from $\mathbb{U}$ into $(-1,1)$).


\section{Some useful examples}

Here we consider mappings related to extremal mappings.

\begin{example}\label{exa:phi}
Let $\varphi$ be the mapping defined by $\displaystyle\varphi(z)=\tan{\left(\frac{\pi}{4}z\right)}$. It is easy to check that the
mapping $\varphi$ is holomorphic and injective on $\mathbb{S}$ and maps $\mathbb{S}$ onto $\mathbb{U}$. Therefore the inverse mapping of the $\varphi$ maps $\mathbb{U}$ onto $\mathbb{S}$. Denote that inverse mapping by $\phi$. It is of interest to consider in more detail the mapping $\phi$. Let
\begin{itemize}
  \item $\phi_{1}(z)=iz$,
  \item $\displaystyle\phi_{2}(z)=\frac{1+z}{1-z}$,
  \item $\phi_{3}(z)=\ln{z}$, where $\ln$ is branch of logarithm  defined on $\{z\in\mathbb{C}:\re{z}>0\}$ and determined by $\ln{1}=0$,
  \item $\displaystyle \phi_{4}(z)=-i\frac{2}{\pi}z$,
\end{itemize}
It is easy to check that $\phi=\phi_{4}\circ\phi_{3}\circ\phi_{2}\circ\phi_{1}$. Also, $\phi(0)=0$.
\end{example}

Let's emphasize that throughout this text by $\varphi$ and $\phi$ we always denote the mappings defined in Example \ref{exa:phi}.

\begin{example}\label{exa:varphia}
For  $a\in \mathbb{U}$, define $\displaystyle \varphi_a(z)= \frac{a-z}{1- \overline{a}z}$. It is well known that $\varphi_a$ is a conformal automorphism of $\mathbb{U}$. Specially, for $a\in(-1,1)$, the mapping $\varphi_a$ has the following properties:
\begin{itemize}
  \item[i)] it is decreasing on $(-1,1)$ and maps $(-1,1)$ onto itself;
  \item[ii)] for $r\in[0,1)$ it holds $\displaystyle \varphi_a([-r,r])=[\varphi_a(r),\varphi_a(-r)]=\left[\frac{a-r}{1-ar},\frac{a+r}{1+ar}\right]$.
\end{itemize}
\end{example}

\begin{example}\label{exa:phib}
Let $b\in\mathbb{S}$ be arbitrary and let $\phi_b$ be a conformal isomorphism from $\mathbb{U}$ onto $\mathbb{S}$ such that $\phi_b(0)=b$   and $\phi_b'(0) >0$. It is straightforward to check that $\phi_b= \phi\circ \varphi_a$, where  $\displaystyle a=\tan\frac{b\pi}{4}$ and $\varphi_a$ is defined in Example \ref{exa:varphia}. Specially, for $b\in(-1,1)$, the mapping $\phi_b$ has the following properties:
\begin{itemize}
  \item[i)]  it is decreasing on $(-1,1)$ and maps $(-1,1)$ onto itself;
  \item[ii)] for $r\in[0,1)$ it holds $\displaystyle \phi_b([-r,r])=[m_b(r),M_b(r)]$, where\\
  $\displaystyle m_b(r)=\phi_b(r)=\frac{4}{\pi}\arctan{\frac{a-r}{1-ar}}$ and $\displaystyle M_b(r)=\phi_b(-r)=\frac{4}{\pi}\arctan{\frac{a+r}{1+ar}}$.
\end{itemize}
\end{example}


\section{Some properties of the strip}

Since $\varphi$ is a conformal isomorphism from $\mathbb{S}$ into $\mathbb{U}$ a simple computation gives
$$
\rho_{\mathbb{S}}(z)=\rho_{\mathbb{U}}(\varphi(z))|\varphi'(z)|=\frac{\pi}{2}\frac{1}{\cos{\left(\displaystyle\frac{\pi}{2}\re{z}\right)}}, \quad \mbox{ for all } \quad z\in\mathbb{S}.
$$
The following proposition gives us an interesting relation between the hyperbolic distance $d_{\mathbb{S}}$ and the Euclidean distance $d_e$. It turns out that this relation is crucial for some of our investigation (see Lemma~\ref{lem:hypdisk1} and Theorems \ref{th:schwhol} and \ref{th:schwhqr} below).

\begin{proposition}\label{prop:strip}
Let $z_1,z_2\in\mathbb{S}$. Then
\begin{equation}\label{strip:fla1}
d_{\mathbb{S}}(z_1,z_2)\geqslant\frac{\pi}{2}d_e(z_1,z_2).
\end{equation}
If $z_1,z_2$ are pure imaginary numbers then in (\ref{strip:fla1}) the equality holds.
\end{proposition}

\begin{proof} Let $\gamma$ be a $C^1$ curve such that joining $z_1$ to $z_2$ in $\mathbb{S}$.  Since $\displaystyle\rho_{\mathbb{S}}(z)\geqslant\frac{\pi}{2}$ for all $z\in\mathbb{S}$, it follows that
\begin{equation}\label{strip:fla2}
      \int_{\gamma}\rho_{\mathbb{S}}(z)|dz|\geqslant\frac{\pi}{2}\int_{\gamma}|dz|.
\end{equation}
In other words the hyperbolic length of the curve $\gamma$ is great or equal to product of $\displaystyle\frac{\pi}{2}$ and Euclidean length of the curve $\gamma$. Since Euclidean length of the curve $\gamma$ is great or equal to $d_e(z_1,z_2)$ according to the inequality (\ref{strip:fla2}), we have
\begin{equation}\label{strip:fla3}
    \int_{\gamma}\rho_{\mathbb{S}}(z)|dz|\geqslant\frac{\pi}{2}d_e(z_1,z_2).
\end{equation}
Take in (\ref{strip:fla3}) infimum over all $C^1$ curves $\gamma$ joining $z_1$ to $z_2$ in $\mathbb{S}$ we obtain $\displaystyle d_{\mathbb{S}}(z_1,z_2)\geqslant\frac{\pi}{2}d_e(z_1,z_2).$

On other hand, let $y_1,y_2\in\mathbb{R}$ be arbitrary and let $\widehat{\gamma}:[0,1]\rightarrow\mathbb{S}$ be defined by $\widehat{\gamma}(t)=iy_1+i(y_2-y_1)t$. An easy computation shows that
\begin{equation}\label{strip:fla4}
   \int_{\widehat{\gamma}}\rho_{\mathbb{S}}(z)|dz|=\frac{\pi}{2}|y_2-y_1|=\frac{\pi}{2}d_e(y_1,y_2).
\end{equation}
\end{proof}

\section{Euclidean properties of hyperbolic discs}
\begin{definition}\label{def1}

Let $\lambda>0$ be arbitrary. By $D_{\lambda}(a)$ (respectively $S_{\lambda}(b)$) we denote the hyperbolic disc in $\mathbb{U}$ (respectively in $\mathbb{S}$) with hyperbolic center $a\in\mathbb{U}$ (respectively $b\in\mathbb{S}$) and hyperbolic radius $\lambda$. More precisely $D_{\lambda}(a)=\{z\in\mathbb{U}:d_{\mathbb{U}}(z,a)<\lambda\}$ and $S_{\lambda}(b)=\{z\in\mathbb{S}:d_{\mathbb{S}}(z,b)<\lambda\}$. Also, $\overline{D}_{\lambda}(a)=\{z\in\mathbb{U}:d_{\mathbb{U}}(z,a)\leqslant\lambda\}$ and $\overline{S}_{\lambda}(b)=\{z\in\mathbb{S}:d_{\mathbb{S}}(z,b)\leqslant\lambda\}$ are corresponding closed discs.  Specially, if $a=0$ (respectively $b=0$) we omit  $a$ (respectively $b$) from the notations.
\end{definition}

\begin{remark}
If $f$ is a conformal isomorphism from $\mathbb{U}$ onto $\mathbb{S}$ such that $f(a)=b$ then $f(D_{\lambda}(a))=S_{\lambda}(b)$ and $f(\overline{D}_{\lambda}(a))=\overline{S}_{\lambda}(b)$.
\end{remark}

Let $r\in(0,1)$ be arbitrary. By $U_r$ we denote Euclidean disc $\{z\in\mathbb{C}: |z|<r\}$ and by $\overline{U}_r$ we denote the corresponding closed disc. Also, let
\begin{equation*}
    \lambda(r)=d_{\mathbb{U}}(r,0)=\ln\frac{1+r}{1-r}=2\artanh{r}.
\end{equation*}
Since $\displaystyle d_{\mathbb{U}}(z,0)=\ln\frac{1+|z|}{1-|z|}=2\artanh{|z|}$ for all $z\in\mathbb{U}$, we have
\begin{equation*}
    D_{\lambda(r)}=\left\{z\in\mathbb{C}:2\artanh{|z|}<2\artanh{r}\right\}=\{z\in\mathbb{C}:|z|<r\}=U_r,
\end{equation*}
and similarly
\begin{equation*}
    \overline{D}_{\lambda(r)}=\overline{U}_r.
\end{equation*}

The closed discs $\overline{D}_{\lambda(r)}$ and $\overline{S}_{\lambda(r)}$ are shown on the Figure \ref{fig:1} and the following lemma to claim that disc $\overline{S}_{\lambda(r)}$ be contained in a Euclidean rectangle.

\begin{figure}[H]
\includegraphics[width=\linewidth]{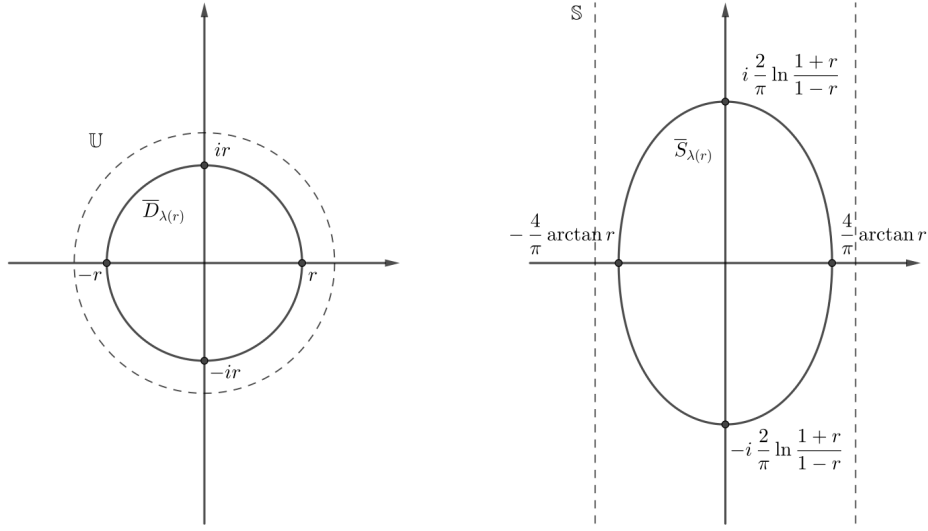}
\caption{$\overline{D}_{\lambda(r)}$ and $\overline{S}_{\lambda(r)}$}\label{fig:1}
\end{figure}

\begin{lemma}\label{lem:hypdisk}
Let $r\in(0,1)$ be arbitrary. Then
\begin{equation}\label{}
\overline{S}_{\lambda(r)}\subset\displaystyle\left[-\frac{4}{\pi}\arctan{r},\frac{4}{\pi}\arctan{r}\right]\times\left[-\frac{2}{\pi}\lambda(r),\frac{2}{\pi}\lambda(r)\right].
\end{equation}
In particular,
\begin{equation}\label{hypdisk:fla0}
    R_{e}(\overline{S}_{\lambda(r)})=\displaystyle\left[-\frac{4}{\pi}\arctan{r},\frac{4}{\pi}\arctan{r}\right].
\end{equation}

\end{lemma}
\begin{proof} Since $\overline{S}_{\lambda(r)}=\phi(\overline{U}_r)$, where $\phi$ is defined in Example \ref{exa:phi}, it is sufficient to show that
\begin{equation}\label{lemhypdisk:fla1}
    \max\{|\re{\phi(z)}|:z\in \overline{U}_r\}=\frac{4}{\pi}\arctan{r}
\end{equation}
and
\begin{equation}\label{lemhypdisk:fla2}
    \max\{|\im{\phi(z)}|:z\in \overline{U}_r\}=\frac{2}{\pi}\lambda(r).
\end{equation}
Let $\phi_1$, $\phi_2$, $\phi_3$ and $\phi_4$ be as in Example \ref{exa:phi}. It is easy to check that $\phi_2\circ\phi_1$ maps  $\partial U_r$ onto $l_r$, where $l_r$ is circle with center $\displaystyle c=\frac{1+r^2}{1-r^2}$ and radius $\displaystyle R=\frac{2r}{1-r^2}$. Also, for all $z\in l_r$ we have $\re(\phi_4(\phi_3(z)))=\displaystyle\frac{2}{\pi}\arg{z}$\footnote{Here $\arg$ is imaginary part of $\ln$, where $\ln$ is branch of logarithm  defined on $\{z\in\mathbb{C}:\re{z}>0\}$ and determined by $\ln{1}=0$. It is evident that values of $\arg$ belong to the interval $\displaystyle\left(-\frac{\pi}{2},\frac{\pi}{2}\right)$.} and $\im(\phi_4(\phi_3(z)))=-\displaystyle\frac{2}{\pi}\ln{|z|}$.

Set $\theta_0=\max\{|\arg{z}|:z\in l_r\}$ and $L_0=\max\{|\ln{|z|}|:z\in l_r\}$.

Let's look at Figure \ref{fig:2}.  It is clear that line $y=(\tan{\theta_0})x$ is a tangent from the point $0$ on the circle $l_r$ and denote by $n_{\theta_0}$ the point of tangency. Also, note that $l_r$ intersect the $x-$axis at the points $\displaystyle c-R=\frac{1-r}{1+r}$ and $\displaystyle c+R=\frac{1+r}{1-r}$ which are reciprocal numbers. Thus, the power of the point $0$ with respect to the circle $l_r$ is equal $1$ and therefore $|n_{\theta_0}|=1$.  Now, it is obviously that $\tan{\theta_0}=R$ and therefore
\begin{equation}\label{lemhypdisk:fla3}
    \theta_0=\arctan{R}=\arctan{\frac{2r}{1-r^2}}=2\arctan{r}.
\end{equation}
Further, since $\ln|n_{\theta_0}|=1$ it is easy seen that $L_0=\max\{-\ln(c-R),\ln(c+R)\}$. But, since $(c-R)(c+R)=1$ it follows that $\ln(c+R)=-\ln(c-R)$ and therefore
\begin{equation}\label{lemhypdisk:fla4}
    L_0=\ln(c+R)=-\ln(c-R)=\ln{\frac{1+r}{1-r}}.
\end{equation}
From (\ref{lemhypdisk:fla3}) and (\ref{lemhypdisk:fla4}) we can get the equalities (\ref{lemhypdisk:fla1}) and (\ref{lemhypdisk:fla2}) (the details are left to the reader).

\begin{figure}[H]
\includegraphics[width=0.75\linewidth]{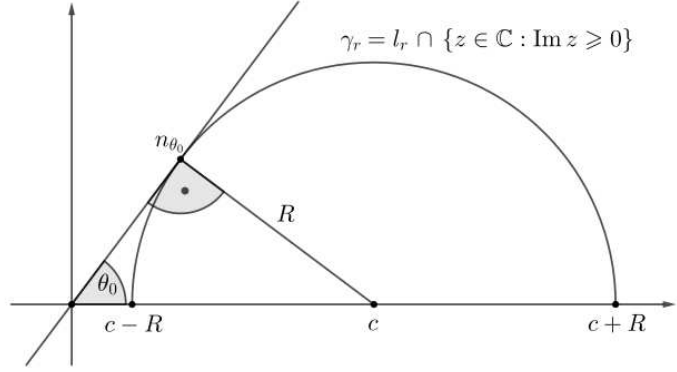}
\caption{Angle $\theta_0$.}\label{fig:2}
\end{figure}
\end{proof}

In order  to appreciate the  proof of the next lemma we give an example.
For $s>0$ set   $R^s= [-1,1]\times [-s,s]$ and  let $\psi^s$ be conformal   mapping of $\mathbb{U}$  onto  $R^s$   such that  $\psi^s$   maps  $(-i,i)$ onto $(-is,is)$  with
$\psi^s(0)=0$.  Also, for $r\in(0,1)$ set $\overline{E}_{r,s} = \psi^s(\overline{U}_r)$. We leave  to the interested  reader to check  that for $r$  enough near to $1$   the function $e$ on
$\overline{E}_{r,s}$  does not attain maximum at  $\psi^s(ir)$.
\begin{lemma}\label{lem:hypdisk1}
Let $\lambda>0$ be arbitrary. Then
\begin{equation}\label{}
    \max\{d_e(z,0):z\in\overline{S}_{\lambda}\}=\frac{2}{\pi}\lambda.
\end{equation}
\end{lemma}
\begin{proof}
Let $z\in\overline{S}_{\lambda}$ be arbitrary. By Proposition \ref{prop:strip} and since $z\in\overline{S}_{\lambda}$ we have
\begin{equation}\label{}
    d_{e}(z,0)\leqslant\frac{2}{\pi}d_{\mathbb{S}}(z,0)\leqslant\frac{2}{\pi}\lambda
\end{equation}
It remains to show that there exists a $z_0\in\overline{S}_{\lambda}$ such that $\displaystyle d_e(z_0,0)=\frac{2}{\pi}\lambda$. Let $\displaystyle z_0=i\frac{2}{\pi}\lambda$ or $\displaystyle z_0=-i\frac{2}{\pi}\lambda$. Then it is clear that $\displaystyle d_e(z_0,0)=\frac{2}{\pi}\lambda$ and by Proposition \ref{prop:strip} we have $d_{\mathbb{S}}(z_0,0)=\lambda$, i.e. $z_0\in\overline{S}_{\lambda}$.

Recall that the above   proof is based on the hyperbolic geometry of  the strip. The reader can try to get a direct  analytic proof without appeal to the geometry.
\end{proof}

\begin{lemma}\label{lem:hypdisk0}
Let $r\in(0,1)$ and $b\in(-1,1)$ be arbitrary. Then
\begin{equation*}
    R_{e}(\overline{S}_{\lambda(r)}(b))=[m_b(r),M_b(r)],
\end{equation*}
where $m_b$ and $M_b$ are defined in Example \ref{exa:phib}.
\end{lemma}

\begin{proof}
Let's repeat that $\overline{D}_{\lambda(r)}=\overline{U}_r$ and $\overline{S}_{\lambda(r)}(b)=\phi_{b}(\overline{D}_{\lambda(r)})=\phi_{b}(\overline{U}_{r})$. Further, one can show that
\begin{equation}\label{hypdisk0:symmetric}
    \overline{S}_{\lambda(r)}(b) \mbox{ is symmetric with respect to the $x$-axis}.
\end{equation}
Also, by  \cite[Theorem 7.11]{BeardonMinda}
\begin{equation}\label{hypdisk0:convex}
    \overline{S}_{\lambda(r)}(b) \mbox{ is Euclidean convex}.
\end{equation}
Now, from (\ref{hypdisk0:symmetric}), (\ref{hypdisk0:convex}) and parts i) and ii) in Example \ref{exa:phib} the lemma follows.
\end{proof}

\section{The Schwarz lemma for harmonic functions from $\mathbb{U}$ into $(-1,1)$}

In this section we first give a simple proof of the classical Schwarz lemma
for harmonic functions and then use the same method we prove a new version (Theorem \ref{th:schwhar1}).

\begin{theorem}[\cite{heinz},{\cite[p. 77]{duren}}]\label{th:schwhar}
Let $u:\mathbb{U}\rightarrow(-1,1)$ be harmonic function such that $u(0)=0$. Then
\begin{equation}\label{schwhar:fla1}
    |u(z)|\leqslant\frac{4}{\pi}\arctan{|z|}, \quad \mbox{ for all } \quad z\in\mathbb{U},
\end{equation}
and this inequality is sharp for each point $z\in\mathbb{U}$.
\end{theorem}

\begin{proof}
Let $z\in\mathbb{U}$ be arbitrary and $r=|z|$. Since $\mathbb{U}$ is simply connected it is well known that there exists $f\in\Hol(\mathbb{U},\mathbb{S})$ such that $u=\re{f}$ and $f(0)=0$. By subordination principle we have $f(\overline{U}_r)\subset\phi(\overline{U}_r)$, where $\phi$ is mapping defined in Example \ref{exa:phi}. Now, since
$\overline{U}_r=\overline{D}_{\lambda(r)}$ and since $\overline{S}_{\lambda(r)}=\phi(\overline{D}_{\lambda(r)})$  by Lemma \ref{lem:hypdisk} we obtain
$u(\overline{U}_r)\subset\displaystyle\left[-\frac{4}{\pi}\arctan{r},\frac{4}{\pi}\arctan{r}\right]$ and the inequality (\ref{schwhar:fla1}) follows.

If  $z=0$ it is clear  that in  (\ref{schwhar:fla1}) the equality holds.
In order to show that the inequality (\ref{schwhar:fla1}) is sharp and for  $z\neq0$, we  define  function  $\widehat{u}:\mathbb{U}\rightarrow(-1,1)$ on the following way $\widehat{u}(\zeta)=(\re{\phi})(e^{-i\arg{z}}\zeta)$\footnote{Here values of $\arg$ belong to the interval $\displaystyle[0,2\pi)$.}, where  $\phi$ is mapping defined in Example \ref{exa:phi}. Note that function $\widehat{u}$ depend on the point $z$. It immediately  follows that $\widehat{u}$ is harmonic function and  $\widehat{u}(0)=0$. A simple computation gives
\begin{equation*}
    |\widehat{u}(z)|=|(\re{\phi})(e^{-i\arg{z}}z)|=|(\re{\phi})(|z|)|=\frac{2}{\pi}\arctan{\frac{2\re|z|}{1-|z|^2}}=\frac{4}{\pi}\arctan{|z|}.
\end{equation*}
%
\end{proof}
We leave   to the interested  reader  to elaborate  proofs of Theorems \ref{th:schwhar} and \ref{th:schwhol}  using  the  Schwarz-Pick lemma (as in Lemma \ref{lem:hqr})  instead of the subordination principle.

\begin{remark}
Using  the  rotation  Theorem \ref{th:schwharcmplx}, stated in the introduction,  follows easily from  Theorem \ref{th:schwhar}.  For details see \cite[p. 77]{duren} cf. also \cite{heinz}.
\end{remark}

\begin{theorem}\label{th:schwhar1}
Let $u:\mathbb{U}\rightarrow(-1,1)$ be harmonic function such that $u(0)=b$ and let $m_b$ and $M_b$ be defined in Example \ref{exa:phib}. Then
\begin{equation}\label{eqX}
m_b(|z|)\leqslant u(z)\leqslant  M_b(|z|), \quad \mbox{for all} \quad z\in\mathbb{U},
\end{equation}
and this inequality is sharp for each point $z\in\mathbb{U}$.
\end{theorem}
\begin{proof}
The proof is analogously to the proof of  Theorem \ref{th:schwhar}. Whereby, instead of the mapping $\phi$ defined in Example \ref{exa:phi} and Lemma \ref{lem:hypdisk} should be used the mapping $\phi_b$ defined in Example \ref{exa:phib} and Lemma \ref{lem:hypdisk0}.
\end{proof}

\section{The Schwarz lemma for holomorphic maps from $\mathbb{U}$ into $\mathbb{S}$}

Theorem \ref{th:schwharcmplx} is usually considered as harmonic version of Theorem \ref{th:schwclassic}. In analogy with Theorems \ref{th:schwclassic} and \ref{th:schwharcmplx} we prove  the next results (Theorems  \ref{th:schwhol} and \ref{th:schwhqr}). Whereby, the  codomain $\mathbb{U}$ and the  function $\arctan$ are replaced by the strip $\mathbb{S}$ and the  function $\artanh$,   respectively.
For $K=1$ Theorem \ref{th:schwhqr} is reduced to Theorem \ref{th:schwhol}.

\begin{theorem}[The Schwarz lemma for holomorphic maps from $\mathbb{U}$ into $\mathbb{S}$]\label{th:schwhol}
Let $f\in\Hol(\mathbb{U},\mathbb{S})$ and $f(0)=0$. Then
\begin{equation}\label{schwhol:ineq1}
    |f(z)|\leqslant\frac{4}{\pi}\artanh{|z|}, \quad \mbox{ for all } \quad z\in\mathbb{U}.
\end{equation}
The inequality (\ref{schwhol:ineq1}) is sharp for each point $z\in\mathbb{U}$. Also,
\begin{equation}\label{schwhol:ineq2}
    |f'(0)|\leqslant\frac{4}{\pi}.
\end{equation}
In (\ref{schwhol:ineq2}) the equality holds if and only if $f(z)=\phi(\alpha z)$, where $\alpha\in\mathbb{C}$ such that $|\alpha|=1$, and $\phi$ is mapping defined in Example \ref{exa:phi}.
\end{theorem}

\begin{proof}
Let $z\in\mathbb{U}$ be arbitrary and $r=|z|$. By subordination principle we have $f(\overline{U}_r)\subset\phi(\overline{U}_r)$, where $\phi$ is mapping defined in Example \ref{exa:phi}. Since
$\overline{U}_r=\overline{D}_{\lambda(r)}$ and since $\overline{S}_{\lambda(r)}=\phi(\overline{D}_{\lambda(r)})$ we have
$f(\overline{U}_r)\subset\overline{S}_{\lambda(r)}$. Hence, by Lemma \ref{lem:hypdisk1} we obtain
\begin{equation}\label{}
    |f(z)|\leqslant\frac{2}{\pi}\lambda(|z|)=\frac{4}{\pi}\artanh{|z|}.
\end{equation}
If  $z=0$ it is clear  that in  (\ref{schwhol:ineq1}) the equality holds.
In order to show that the inequality (\ref{schwhol:ineq1}) is sharp and for  $z\neq0$, we  define  function  $\widehat{f}:\mathbb{U}\rightarrow\mathbb{S}$ on the following way $\widehat{f}(\zeta)=\phi(ie^{-i\arg{z}}\zeta)$\footnote{Here values of $\arg$ belong to the interval $\displaystyle[0,2\pi)$.}, where  $\phi$ is  defined in Example \ref{exa:phi}. Note that function $\widehat{f}$ depend on the point $z$. It immediately  follows that $\widehat{f}\in\Hol(\mathbb{U},\mathbb{S})$ and  $\widehat{f}(0)=0$. A simple computation gives
\begin{equation*}
    |\widehat{f}(z)|=|\phi(ie^{-i\arg{z}}z)|=|\phi(i|z|)|=\left|-i\frac{2}{\pi}\ln\frac{1-|z|}{1+|z|}\right|=\frac{4}{\pi}\artanh{|z|}.
\end{equation*}
Finally, by subordination principle we obtain $\displaystyle|f'(0)|\leqslant|\phi'(0)|=\frac{4}{\pi}$ and theorem follows.

%
%
\end{proof}

\section{The Schwarz lemma for harmonic K-quasiregular maps from $\mathbb{U}$ into $\mathbb{S}$}
Quasiregular maps are a class of continuous maps between Euclidean spaces $\mathbb{R}^n$ of the same dimension or, more generally, between Riemannian manifolds of the same dimension, which share some of the basic properties with holomorphic functions of one complex variable.

Let $D$ and $G$ be domains in $\mathbb{C}$. A $C^1$ mapping $f:D\rightarrow G$ we call sense-preserving $K-$quasiregular mapping if
\begin{itemize}
  \item[a)] $|f_z(z)|>|f_{\overline{z}}(z)|$ for all $z\in D$;\\
  \item[b)] there exists $K\geqslant1$ such that $\displaystyle\frac{|f_z(z)|+|f_{\overline{z}}(z)|}{|f_z(z)|-|f_{\overline{z}}(z)|}\leqslant K$ for all $z\in D$.
\end{itemize}

Thus the linear map $(df)_z=f_z(z)dz+f_{\overline{z}}(z)d\overline{z}$  maps circles with center at $z$  onto ellipses such that the ratio  between the big axis and the small axis is uniformly bounded by $K$ with respect to $z\in D$.

Injective  $K-$quasiregular mappings are called  $K-$quasiconformal  mappings.
Quasiconformal maps  play a crucial role in  Teichm\"{u}ller theory and complex dynamics.

The class of all harmonic sense-preserving $K-$quasiregular mapping $f:D\rightarrow G$ we denote by $\HQR_{K}(D,G)$.

\begin{example}\label{exa:psiK}
Let $K\geqslant1$ and let $A_{K}:\mathbb{S}\rightarrow\mathbb{S}$ be defined by $A_{K}(x,y)=(x,Ky)$. It is clear that the mapping $A_{K}$ is sense-preserving $K-$quasiregular. Let $\psi_K=A_{K}\circ\phi$, where $\phi$ is the mapping defined in Example 1. It is easy to check that $\psi_K\in\HQR_{K}(\mathbb{U},\mathbb{S})$.
\end{example}

\begin{lemma}\label{lem:hqr}
Let $K\geqslant1$, $f\in\HQR_K(\mathbb{U},\mathbb{S})$. Then
\begin{equation}\label{lemhqr:fla1}
 d_{\mathbb{S}}(f(z_1),f(z_2))\leqslant Kd_{\mathbb{U}}(z_1,z_2) \quad  \mbox{ for all }\quad z_1,z_2\in\mathbb{U}.
\end{equation}
\end{lemma}
\begin{proof}
Set $u=\re{f}$ and $\nabla u=(u_x,u_y)$. Since $f$ is $K-$quasiregular one can check that
\begin{equation}\label{lemhqr:fla2}
     |f_z(z)|+|f_{\overline{z}}(z)|\leqslant K|\nabla u(z)| \quad \mbox{ for all } \quad z\in\mathbb{U}.
\end{equation}
By Proposition \ref{prop:mm} we have
\begin{equation}\label{lemhqr:fla3}
    \rho_{\mathbb{S}}(f(z))|\nabla u(z)|\leqslant \rho_{\mathbb{U}}(z)\quad \mbox{ for all }\quad z\in\mathbb{U}.
\end{equation}

From (\ref{lemhqr:fla2}) and (\ref{lemhqr:fla3}) it follows that
\begin{equation}\label{}
    \rho_{\mathbb{S}}(f(z))\left(|f_z(z)|+|f_{\overline{z}}(z)|\right)\leqslant K\rho_{\mathbb{U}}(z)\quad \mbox{ for all }\quad z\in\mathbb{U}.
\end{equation}
It is well known  in general   that  the estimate of the gradient by
means of the corresponding densities yields the corresponding  estimate
between the distances, see for example \cite{MMSchw_Kob}. The detailed verification of it  being left to the
reader. In particular, we get (\ref{lemhqr:fla1}).
\end{proof}

Note that if codomain is  $\mathbb{U}$   the result of this type is proved in \cite{MaK} and \cite{MKnez}.

\begin{theorem}[The Schwarz lemma for HQR maps from $\mathbb{U}$ into $\mathbb{S}$]\label{th:schwhqr}
Let $K\geqslant1$, $f\in\HQR_K(\mathbb{U},\mathbb{S})$ and $f(0)=0$. Then
\begin{equation}\label{schwhqr:fla}
    |f(z)|\leqslant\frac{4}{\pi}K\artanh{|z|}, \quad \mbox{ for all } \quad z\in\mathbb{U},
\end{equation}
and this inequality is sharp for each point $z\in\mathbb{U}$.
\end{theorem}

\begin{proof}
Let $z\in\mathbb{U}$ be arbitrary. By the Lemma \ref{lem:hqr} we have
\begin{equation}\label{schwhqr:fla1}
    d_{\mathbb{S}}(f(z),0)\leqslant Kd_{\mathbb{U}}(z,0).
\end{equation}
Since $d_{\mathbb{U}}(z,0)=\lambda(|z|)$, from (\ref{schwhqr:fla1}) it follows that $f(z)$  belongs to the closed hyperbolic disc
with hyperbolic center $0$ and hyperbolic radius $K\lambda(|z|)$, i.e. $f(z)~\in~\overline{S}_{K\lambda(|z|)}$.
Hence, by Lemma~\ref{lem:hypdisk1} we get $\displaystyle |f(z)|\leqslant\frac{2}{\pi}K\lambda(|z|)=\frac{4}{\pi}K\artanh{|z|}$.

As in the proof of Theorem \ref{th:schwhol}, one can show that the inequality (\ref{schwhqr:fla}) is sharp. In this case, instead of mapping $\phi$ defined in Example $\ref{exa:phi}$ the mapping $\psi_K$ defined in Example $\ref{exa:psiK}$ should be used.


\end{proof}
{\it Acknowledgement}. The authors  discussed  the results of these types with members of \emph{Belgrade analysis seminar}, B. Karapetrovi\'c, N. Mutavd\v zi\'c and B. Jevti\'c, who also obtained
some results  related to this subject. We are indebted  to the above mention  colleagues
for useful   discussions, and plan in forthcoming paper (hopefully with them) to discuss  further progress.


\end{document}